\documentclass{article}
\usepackage{authblk}
\usepackage{geometry}
\usepackage{graphicx,xcolor} 
\usepackage[bf,font=small]{caption}
\usepackage{url}
\usepackage{amsfonts,amssymb,amsmath,amsthm,mathtools,thmtools, thm-restate, mathrsfs}
\declaretheorem[name=Theorem]{theorem}
\declaretheorem[name=Lemma, sibling=theorem]{lemma}
\declaretheorem[name=Corollary, sibling=theorem]{corollary}
\declaretheorem[name=Conjecture, sibling=theorem]{conjecture}

\newcommand{\equivgraph}{$1_{k}$}

\title{On Modular Edge Colourings of Graphs}
\author[1]{Gaétan Berthe}
\author[2]{Marthe Bonamy}
\author[3]{Fábio Botler}
\author[4]{Gaia Carenini}
\author[3]{Lucas Colucci}
\author[5]{Arthur Dumas}
\author[6]{Fatemeh Ghasemi}
\author[3]{Pedro Mariano Viana Neto}

\affil[1]{Université Clermont Auvergne, CNRS, Clermont Auvergne INP, Mines Saint-Étienne, LIMOS, 63000 Clermont-Ferrand, France.}
\affil[2]{CNRS, LaBRI, Université de Bordeaux, Bordeaux, France.}
\affil[3]{Instituto de Matemática e Estatística, Universidade de São Paulo,
    Brazil.}
\affil[4]{DPMMS, University of Cambridge, Cambridge, United Kingdom.}
\affil[5]{Univ Lyon, CNRS, INSA Lyon, UCBL, Centrale Lyon, Univ Lyon 2, LIRIS, UMR5205, F-69622 Villeurbanne, France.}
\affil[6]{LACL, Université Paris-Est Créteil, France.}

\date{\today}
\begin{document}

\maketitle

\begin{abstract}
Given a graph $G$ and an integer $k\geq 2$, let $\chi'_k(G)$ denote the minimum number of colours required to colour the edges of $G$ such that, in each colour class, the subgraph induced by the edges of that colour has all non-zero degrees congruent to $1$ modulo $k$. In 1992, Pyber 
  proved that \(\chi'_2(G) \leq 4\) for every graph \(G\), and posed the question of whether \(\chi'_k(G)\)
  can be bounded solely in terms of \(k\) for every \(k\geq 3\). This question was answered in 1997 by Scott, who showed that $\chi'_k(G)\leq5k^2\log k$,
  and further asked whether \(\chi'_k(G) = O(k)\). Recently, Botler, Colucci, and Kohayakawa (2023) answered Scott's question affirmatively proving that \(\chi'_k(G) \leq 198k - 101\),
  and conjectured that the multiplicative constant could be reduced to \(1\). A step towards this latter conjecture was made in 2024 by Nweit and Yang, who improved the bound to \(\chi'_k(G) \leq 177k - 93\).
  In this paper, we further improve the multiplicative constant to~\(9\). 
  More specifically, we prove that there is a function \(f\in o(k)\) for which \(\chi'_k(G) \leq 7k + f(k)\) if \(k\) is odd,
  and \(\chi'_k(G) \leq 9k + f(k)\) if \(k\) is even. 
  In doing so, we prove that \(\chi'_k(G) \leq k + O(d)\) for every \(d\)-degenerate graph $G$, which plays a central role in our proof. 

\end{abstract}

%DIF <  \newpage
\section{Introduction}

The study of graph colourings with modular constraints has a well-established history, with numerous classical results in the area. According to Lovász \cite[Exercise 5.17]{lovasz2007combinatorial}, one such result is due to Gallai and states that every graph admits a bipartition of its vertex set in which each part induces an even graph, i.e., a graph in which each vertex has even degree. Over the past few decades, various extensions and refinements of this theorem have attracted considerable attention. Among these, perhaps the most extensively studied, is the problem of partitioning the vertex set of a graph into parts each inducing an odd subgraph, i.e., a subgraph in which every vertex has odd degree \cite{ferber2022every,scott1992large, scott2001induced}. When we go beyond even and odd, i.e., by requiring each part to induce a subgraph whose degrees have prescribed residues modulo \(k\), this problem is widely open. For example, a widely circulated folklore conjecture asserts that it suffices a constant number (depending on \(k\)) of colours to partition the vertices of a graph so that each colour class induces a graph in which all of the degrees are \(0\) modulo \(k\)
\cite[Problem 2]{caro1994zero}. At the moment, very few results are known around this conjecture
\cite{balister2023counting,ferber2023subgraphs, hunter2023result}. 

In this work, we turn to a natural analogue of these problems in the setting of edge colourings. In what follows, unless otherwise specified, \(k \geq 2\) denotes an integer. All graphs considered are simple, and \(e(G)\) denotes the number of edges in a graph \(G\). A \emph{\(\chi'_k\)-colouring} of \(G\) is an edge-colouring of \(G\) in which each colour class, i.e., the subgraph induced by its edges, has all vertex degrees congruent to \(1 \pmod{k}\). 
The \emph{mod $k$ chromatic index of \(G\)}, denoted by \(\chi'_k(G)\), is the minimum number of colours required for a \(\chi'_k\)-colouring of \(G\). Observe that any upper bound on  $\chi'_k(G)$ that depends only on $k$ immediately implies the aforementioned conjecture for line graphs.

The study of \(\chi'_k(G)\) was initiated by Pyber~\cite{pyber1991covering} proved that $\chi'_2(G) \leq 4$ for every graph $G$ and asked whether
$\chi'_k(G)$ is bounded by some function of \(k\) only. Scott
\cite{scott1997graph} proved that $\chi'_k(G) \leq 5k^2\log k$ for
every graph \(G\), and in turn asked if $\chi'_k(G)$ is in fact
bounded by a linear function of $k$. This would be the best possible apart from the multiplicative constant,
as \(\chi'_k(K_{1,k})=k\).  
Such a question was answered affirmatively by Botler, Colucci and Kohayakawa in 2023~\cite{Modk-198}, who proved that \(\chi'_k(G) \leq 198k - 101\) for every graph \(G\). They furthermore conjectured that the constant \(198\) can be brought all the way down to \(1\).

\begin{conjecture}[Botler--Colucci--Kohayakawa, 2023]\label{conj:main}
  Let \(k \geq 2\) be an integer.
  There is a constant $C$ such that $\chi'_k(G) \leq k+C$ for every
  graph~$G$.
\end{conjecture}

In support of Conjecture~\ref{conj:main}, Botler, Colucci, and Kohayakawa~\cite{modk-random}
proved that for every positive integer \(k\) there is a constant \(C_k\) such that if \(np \geq C_k\log k\) and \(n(1-p) \to \infty\) as \(n\to\infty\), 
then the following holds asymptotically almost surely: if \(k\) is odd, 
then \(\chi'_k(G(n,p)) = k\),
while if \(k\) is even, 
then \(\chi'_k(G(2n,p)) = k\) and \(\chi'_k(G(2n+1,p)) = k+1\). They also observed that the constant~$C$ in Conjecture~\ref{conj:main} must be at least~$2$, since \(\chi'_k(K_{1,k,k}) = k+2\). Further progress towards verifying Conjecture ~\ref{conj:main} was made by Nweit and Yang \cite{nweit2024mod}, who improved the multiplicative constant to \(177\), proving that \(\chi'_k(G) \leq 177k - 93\). In this paper, we further improve this constant as follows. 

\begin{restatable}{theorem}{mainTh}
\label{th:main}
There is a function \(f \in o(k)\)
such that for every positive integer $k$ and every graph $G$, we have $\chi'_k(G) \leq 
9k+f(k)$.
In addition, if $k$ is odd then $\chi'_k(G) \leq 
7k+f(k)$.
\end{restatable}

The proof of Theorem~\ref{th:main} follows a general approach~\cite{Modk-198,nweit2024mod,scott1997graph} that involves removing a maximal colour class -- that is, the largest subgraph whose degrees are all congruent to \(1 \pmod{k}\) -- to reduce the problem to a graph with bounded degeneracy, and then constructing a \(\chi'_k\)-colouring of such graph using few colours. Our improvement builds on two key ingredients. First, we establish a stronger upper bound for the weak Alon--Friedland--Kalai Conjecture (see Conjecture~\ref{conj:alonfriedlandkalai}), which yields a tighter estimate on the number of edges -- and consequently, on the degeneracy -- of the graph remaining after removing a maximal colour class. Second, we provide an improved upper bound on the mod \(k\) chromatic index of graphs with bounded degeneracy. 
This is our main technical contribution.
\medskip

\noindent\textbf{The Alon--Friedland--Kalai Conjecture.}
We say that a graph \(H\) is \emph{\(k\)-divisible}
if \(d_H(u) \equiv 0 \pmod{k}\) for every \(u\in V(H)\). In 1984, Alon, Friedland, and Kalai~\cite{AlonFriedlandKalai-RegularSubgraphs} 
explored the extremal problem of calculating the maximum number of edges of a graph on \(n\) vertices with no \(k\)-divisible subgraph,
and proved the following upper bound 
in the case \(k\) is a prime power.
\begin{theorem}[Alon--Friedland--Kalai, 1984]\label{thm:alonfriedlandkalai}
    Let \(k\) be a prime power.
    If \(G\) is a graph on \(n\) vertices
    that does not contain a non-empty \(k\)-divisible subgraph,
    then \(e(G) \leq (k-1)n\).
\end{theorem}
They also conjectured that Theorem~\ref{thm:alonfriedlandkalai} can be generalized for any \(k\).
\begin{conjecture}[Alon--Friedland--Kalai, 1984]\label{conj:alonfriedlandkalai}
    Let \(k\) be a positive integer.
    If \(G\) is a graph on \(n\) vertices
    that does not contain a non-empty \(k\)-divisible subgraph,
    then \(e(G) \leq (k-1)n\).
\end{conjecture}
Botler, Colucci, and Kohayakawa proved a weaker version of Conjecture~\ref{conj:alonfriedlandkalai} in which \((k-1)n\)
is replaced by $(24k-12)n$~\cite[Lemma~3]{Modk-198}. 
We show that this bound can be sharpened as follows. 
Throughout the text, given a positive integer \(k\),
we denote by \(q(k)\) the smallest prime number that is equal to, or greater than $k$.
\begin{lemma}\label{lem: k-div}
Let $k\geq 2$ be an integer.  If $G$ is an $n$-vertex graph that does not contain any non-empty $k$-divisible subgraph, then $e(G)\leq (q(\frac{3k}{2})-1)n$. Moreover, if $k$ is odd, then $e(G)\leq (q(k)-1)n$. 
\end{lemma}

The proof of Lemma \ref{lem: k-div} is fairly straightforward and combines a simple idea from ~\cite{Modk-38} with known results guaranteeing the existence of \(k\)-regular subgraphs in \(r\)-regular graphs~\cite{even-regular,tashkinov1982regular}.

\begin{lemma}
[Tashkinov, 1982]
\label{lemma:Tashkinov-1}
Let \(r\) and \(k\) be positive odd integers
with $0< k < r$.
Then every $r$-regular graph contains a $k$-regular subgraph.
\end{lemma}

\begin{lemma}
[Kano, 1989]
\label{lemma:kano-1}
Let \(r\) and \(k\) be respectively 
positive odd and even integers with $2\leq k \leq 2r/3$.
Then every $r$-regular graph contains a $k$-regular subgraph.
\end{lemma}

\begin{proof}[Proof of Lemma \ref{lem: k-div}]
First, suppose that $k$ is even and that $e(G)>(q(\frac{3k}{2}) - 1)n$. 
By Theorem \ref{thm:alonfriedlandkalai}, $G$ contains a $q(\frac{3k}{2})$-divisible subgraph. 
Let $G'$ be a $q(\frac{3k}{2})$-regular graph obtained from $G$ by splitting each vertex $u\in V(G)$ of degree $\ell \cdot q(\frac{3k}{2})$ into $\ell$ distinct vertices of degree $q(\frac{3k}{2})$, say \(u_1,\ldots u_\ell\), so that \(N(u) = \cup_{i=1}^\ell N(u_i)\). 
Now, by Lemma~\ref{lemma:kano-1}, \(G'\) contains a $k$-regular subgraph \(H'\). 
Let \(H\) be the subgraph of \(G\) 
corresponding to \(H'\),
i.e., obtained from \(H'\) by identifying the vertices \(u_1,\ldots, u_\ell\) into \(u\),
and observe that \(H\) is \(k\)-divisible because \(H'\) is \(k\)-regular.
The proof in the case of odd \(k\) is analogous,
but uses \(q(k)\) instead of \(q(\frac{3k}{2})\)
and Lemma~\ref{lemma:Tashkinov-1} instead of Lemma~\ref{lemma:kano-1}. 
\end{proof}

A result of Baker, Harman, and Pintz~\cite{baker2001difference} states that for all sufficiently large integers $k$, we have $q(k)\leq k+k^{0.6}$. Therefore, we can deduce the following corollary. 

\begin{corollary}\label{cor: edges}
Given a positive integer \(k\), if \(G\) is a graph on \(n\) vertices without non-empty \(k\)-divisible subgraphs,
then \(e(G) \leq (3/2 + o(1))kn\).   
\end{corollary}

\section{Proof of Theorem~\ref{th:main}}
In this section, we prove Theorem~\ref{th:main}.
First, we deal with the case of degenerate graphs.
Given an integer $d$, we say that a $n$-vertex graph is $d$-\emph{degenerate} if there is an ordering $v_1,\dots,v_n$ of its vertices for which the number of neighbours of $v_i$ in $\{v_1,\dots, v_{i-1}\}$ is at most $d$. 
In 2023, Botler, Colucci, and Kohayakawa proved that
\(\chi'_k(G) \leq 4d + 2k - 2\) 
for every \(d\)-degenerate graph~\cite[Lemma~4]{Modk-198}. We improve their result as follows. 

\begin{lemma}\label{lem:deg}
        Let \(k\), \(d\), and \(a\) be positive integers 
        with \(k \geq 3\).
        If \(G\) is a \((d+1)\)-degenerate graph, then 
        $$\chi'_k(G)\leq 2d+k+a + d/a + 1.$$
\end{lemma}

By setting \(a = \lfloor\sqrt{d}\rfloor\), Lemma~\ref{lem:deg}
states that \(\chi'_k(G) \leq 2d + k + 2\lfloor\sqrt{d}\rfloor + 3\) for every \((d+1)\)-degenerate graph \(G\).

 Our starting point to prove Lemma \ref{lem:deg} is the following simple Hall-type result. 
Given disjoint sets \(U\) and \(W\),
a \emph{\((U,W)\)-bipartite graph} is a bipartite graph with vertex partition~\((U,W)\).

\begin{lemma}\label{lemma:matching}
    Let \(k\) be a positive integer
    and let \(G\) be a \((B,X)\)-bipartite graph
    such that \(d(u) \leq k\) for every \(u\in B\),
    and \(d(v) \geq k\) for every \(v\in X\).
    Then \(G\) has a matching that covers \(X\).
\end{lemma}

\begin{proof}
    We check Hall's condition.
    Let \(S\subseteq X\).
    We count \(e(S,N(S))\).
    Since \(d(u) \leq k\) for every \(u\in B\),
    we have \(e(S,N(S)) \leq k \cdot|N(S)|\);
    and since \(d(v) \geq k\) for every \(v\in X\),
    we have \(e(S,N(S)) \geq k\cdot|S|\).
    Therefore
    \[
    k\cdot|N(S)| \geq e(S,N(S)) \geq k\cdot|S|,
    \]
    which gives us \(|N(S)|\geq |S|\) as desired.
\end{proof}

Our proof of Lemma \ref{lem:deg} consists of consecutive applications of the following more general lemma on partitioning the vertex set of a bipartite graph into special stars.
In what follows, we say that a star \(S\) is an \emph{\(\ell_k\)-star}
if \(e(S) = \ell \pmod{k}\),
and we say that \(S\) is a \emph{trivial} \(\ell_k\)-star if \(e(S) = \ell\).

\begin{lemma}\label{lemma:star-covering}
    Let \(k\), \(d\), and \(a\) be positive integers,
    and let \(A\) and \(B\) be disjoint sets such that
    \(|A| \leq d+1\) and \(|A| + |B| = 2d + k + a + d/a + 1\).
    Let \(G\) be a \((A\cup B, X)\)-bipartite graph, such that \(d(x) \geq |A\cup B| - d\) for every \(x\in X\).
    Then, there is a covering of \(X\) 
    by vertex-disjoint stars \(\mathcal{S}_1,\ldots,\mathcal{S}_s, \mathcal{S}'_1,\ldots,\mathcal{S}'_{s'}\)
    such that \(\mathcal{S}_i\) is a \(0_k\)-star centered in \(A\) for every \(i\in [s]\),
    and \(\mathcal{S}'_i\) is an \(1_k\)-star centered in \(B\) for every \(i\in [s']\).
\end{lemma}

\begin{proof}
    Let \(A'\subseteq A\cup B\) be a set of size \(d + a\) 
    that contains \(A\) (this is possible because \(|A|\leq d+1\)), 
    and let \(B' = (A\cup B)\setminus A'\).
    First, while possible let \(S_i\) be a maximum \(0_k\)-star centered in \(A\)
    or a maximum non-trivial \(1_k\)-star centered in \(A'\setminus A\) (equivalently, $A'\cap B$),
    so that the obtained stars \(S_1,\ldots, S_r\) have disjoint leaves.
    Let \(X'\subseteq X\) be the set of vertices obtained from \(X\) by removing the leaves of such stars.
    Observe that \(d_{X'}(u) \leq k\) for every \(u\in A'\),
    otherwise we could pick one more such star, or increase an existing star.
    Also, note that since \(d(x) \geq |A\cup B| - d\), we have that \(d_{A'}(x) \geq a\) holds for every \(x\in X'\).
    Therefore, we have
    \[
        a\cdot |X'| \leq e(X',A') \leq k\cdot |A'| = k\cdot (d+a),
    \]
    which gives us \(|X'| \leq k + d\cdot k/a\).

    Now, while possible pick a maximum non-trivial \(1_k\)-star \(S'_i\) centered in \(B'\) whose leaves are in \(X'\)
    so that the obtained stars \(S'_1,\ldots, S'_{t'}\) have disjoint leaves.
    Let \(X''\subseteq X'\) be the set of vertices obtained from \(X'\) by removing the leaves of such stars.
    Analogously to the step above, 
    we have \(d_{X''}(u) \leq k\) for every \(u\in B'\),
    otherwise we could pick one more such star, or increase an existing star.
    Observe that each such star used at least \(k+1\) vertices of \(X'\),
    and hence, 
    \[
        |X'| \geq (k+1) \cdot t' \geq k \cdot t',
    \]
    which gives us \(t' \leq 1 + d/a\).

    Now, let \(B''\subseteq B'\) be the set obtained from \(B'\)
    by removing the centres of \(S'_1,\ldots, S'_{t'}\).
    Note that \(|B''| = |B'| - t' \geq d + k\).
    Therefore, \(d_{B''}(x) \geq k\) for every vertex \(x\in X''\),
    and hence \(H = G[B''\cup X'']\) is a bipartite graph
    for which \(d_{X''}(u) \leq k\) for every \(u\in B''\)
    and \(d_{B''}(x) \geq k\) for every vertex \(x\in X''\),
    and, by Lemma~\ref{lemma:matching},
    there is a matching \(M\) that covers \(X''\).
   It follows that each edge in \(M\) 
    is a \(1_k\)-star centred in \(B''\). This completes the proof since the collection of stars \(\mathcal{S}_1,\ldots,\mathcal{S}_{s}, \mathcal{S}'_1,\ldots,\mathcal{S}'_{s'}\) is given by the union of the collections of stars identified before.
    \end{proof}
 We can now prove our main lemma.
 Given non-negative integers \(\ell\) and \(k\), with \(\ell < k\),
 we say that a graph \(G\) is an \emph{\(\ell_k\)-graph}
 if every vertex of \(G\) has degree \(\ell\) modulo \(k\).

    \begin{proof}[Proof of Lemma \ref{lem:deg}]
        Let $G$ be a $(d+1)$-degenerate graph and let $v_1, \ldots, v_n$ be an ordering of its vertices so that for every $ 1 \leq i \leq n$, the vertex $v_i$ has at most $d+1$ neighbours in $\{v_1,\ldots,v_{i-1}\}$. 
        We describe a \(\chi'_k\)-colouring of $G$ that we define recursively while maintaining the following invariant: at the beginning of step $i$, all the edges incident to a vertex $v_j$ with $j<i$ are coloured in such a way that each colour class induces a \equivgraph-graph, and no other edge is coloured. 
        Moreover, for each \(j\geq i\),
        all the coloured edges incident to \(v_j\) have distinct colours.
        This invariant is trivially satisfied at the beginning of step $1$.

        Therefore, assume that the first $i-1$ steps have been carried out, and  denote by $c_i$ the obtained colouring. 
        At step $i$, we consider the vertex $v_i$. 
        By construction $v_i$ is incident to at most $d+1$ already coloured edges. 
        We extend $c_i$ to all the edges of type $v_iv_j$ with $j>i$ while keeping the invariant. 
        For that, we partition the set of all colours into two sets $A$ and $B$,
        where \(A\) consists of the colours appearing on an edge incident to $v_i$.
        Let \(X\) be the set of neighbors of \(v_i\)
        whose edge to \(v_i\) is still uncoloured,
        and let \(H\) be the auxiliary \((A\cup B, X)\)-bipartite graph 
        in which \(x\in X\) is adjacent to \(u\in A\cup B\)
        if no edge incident to \(x\) is coloured with colour \(u\).
        Since \(G\) is \((d+1)\)-degenerate,
        we have \(d_H(x) \geq |A\cup B| - d\) for every \(x\in X\).
        By Lemma~\ref{lemma:star-covering},
        \(X\) can be covered (in \(H\)) by vertex disjoint \(0_k\)-stars centered in \(A\) and \(1_k\)-stars centered in \(B\).
        Such stars induce the desired colouring.
    \end{proof} 

\paragraph{Putting everything together}
We can now show how Lemmas~\ref{lem: k-div} and~\ref{lem:deg} imply Theorem \ref{th:main}. 

\begin{proof}[Proof of Theorem~\ref{th:main}]
	Let $H$ be a maximal \equivgraph-subgraph of \(G\), 
    and let $G'=(V(G), E(G)\setminus E(H))$. 
    Note that if there is an edge \(e\) joining two vertices in $V(G)\setminus V(H)$, then $H'=H+e$ is a \equivgraph-subgraph of \(G\) with more edges than \(H\), a contradiction to the maximality of \(H\). 
    Therefore, $V(G)\setminus V(H)$ is an independent set in $G'$. 
    Similarly, the graph $G'[V(H)]$ does not contain a non-empty $k$-divisible subgraph,
    otherwise, one could add the edges of such a subgraph to $H$ and obtain a larger \equivgraph-subgraph of \(G\).
    Lastly, it is easy to see that every vertex in $V(H)$ has at most $k-1$ neighbours in $V(G)\setminus V(H)$,
    otherwise we could add \(k\) of these edges to \(H\) 
    and obtain a larger \equivgraph-subgraph of \(G\).

    We can now describe a \(\chi'_k\)-colouring of $G$. 
    We colour \(H\) with colour \(1\).
	% First, we colour the edges of $H$. This can be done using a unique colour. 
    % The remaining edges to be coloured are the edges of $G'$ which are all incident to $V(H)$.	
    We distinguish two cases according to the parity of $k$.
    First, suppose \(k\) is odd. 
    By Lemma \ref{lem: k-div}, every subgraph $H'\subseteq G'[V(H)]$ has at most $(q(k)-1)|V(H')|$ edges, and hence has minimum degree at most $2(q(k)-1)$. 
    Consequently, every subgraph of \(G'\) has minimum degree at most $d = 2(q(k)-1)+k-1 = 2q(k) + k -3$,
    and hence $G'$ is $(2q(k)+k-3)$-degenerate.
   By Lemma~\ref{lem:deg} with \(a = \lfloor\sqrt{d}\rfloor\), we have 
    $$\chi'_k(G)\leq 2(2q(k)+k-3) + k + 2\lfloor\sqrt{d}\rfloor + 4\leq 7k+f(k),$$
    where we used that $q(k)\leq k+k^{0.6}$.

   The case where \(k\) is even is analogous. 
   However, due to the difference between the bounds given by Lemma~\ref{lem: k-div},
   in this case, $G'$ is $(2q(\frac{3k}{2})+k-3)$-degenerate. Consequently, $\chi'_k(G)\leq 9k+f(k)$. This concludes the proof. 
\end{proof}

\section{Concluding remarks}
It is natural to ask whether the current approach can be further refined to yield a better constant. In what follows, we explore this possibility and highlight the main obstacles that must be overcome. Two types of improvements are possible without departing from the overall strategy.  The first involves tightening the upper bound on the number of edges in a graph on \(n\) vertices with no \(k\)-divisible subgraph,
i.e., making further progress towards the Alon--Friedland--Kalai conjecture (Conjecture \ref{conj:alonfriedlandkalai}).  The second concerns improving the bounds on the number of colours needed for \(\chi'_k\)-colourings of $d$-degenerate graphs. 

Making progress toward the Alon--Friedland--Kalai conjecture appears challenging, but it remains a compelling direction with implications that extend well beyond \(\chi'_k\)-colourings. From a quantitative perspective, a proof of the Alon--Friedland--Kalai Conjecture would, even without any further refinements to the argument, immediately yield the bound $\chi'_k(G)\leq 7k+o(k)$ for all $k$. 

While improving the number of colours required for \(\chi'_k\)-colourings of $d$-degenerate graphs may seem more tractable, it nonetheless appears to demand new ideas, as we discuss next.
Let $G$ be a $d$-degenerate graph and let $v_1, \ldots, v_n$ be an ordering of its vertices such that, for every $ 1 \leq i \leq n$, the vertex $v_i$ has at most $d$ neighbours in $\{v_1,\ldots,v_{i-1}\}$. All \(\chi'_k\)-colourings of $G$ considered so far have been constructed recursively while maintaining the following invariant: at the beginning of step $i$, all edges incident to a vertex $v_j$ with $j<i$ are coloured so that each colour class induces a \equivgraph-graph, and no other edge is coloured. 
Moreover, all coloured edges incident to \(v_i\) are coloured with distinct colours.
We show that under such a colouring strategy, no substantial improvement in the number of colours can be achieved. To this end, we prove that Lemma~\ref{lemma:star-covering} is tight up to the~\(\sqrt{d}\) term. 
\begin{lemma}
     Let \(k\) and \(d\) be positive integers,
    and let \(A\) and \(B\) be disjoint sets of vertices such that
    \(|A| = d+1\) and \(|A| + |B| = 2d + k - 1\).
    Let \(B^*\subseteq B\) be a set of size \(d\),
    and let \(G\) be the \((A\cup B, X)\)-bipartite graph
    such that \(x\) is adjacent to every vertex in \(A\cup (B\setminus B^*)\)
    for every \(x\in X\).
    If \(|X| \equiv k-1\pmod{k}\),
    then \(G\) has no covering of \(X\) 
    by vertex-disjoint stars \(S_1,\ldots,S_r, S'_1,\ldots,S'_t\)
    such that \(S_i\) is a \(0_k\)-star centred in \(A\) for every \(i\in [r]\),
    and \(S'_i\) is an \(1_k\)-star centred in \(B\) for every~\(i\in [t]\).   
\end{lemma}

\begin{proof}
    Suppose that there is a covering \(S_1,\ldots,S_r, S'_1,\ldots,S'_t\) of \(X\)
    as above.
    Given a star \(S\) let \(l(S)\) be the number of leaves of \(S\).
    Then we have
    \(
        |X| = \sum_{i=1}^r l(S_i) + \sum_{i=1}^t l(S'_i)
    \),
    and hence, since \(S_i\) is a \(0_k\)-star for every \(i\in[r]\)
    and \(S'_i\) is a \(1_k\)-star for every \(i\in[t]\),
    \[
        k-1 \equiv |X| \equiv \sum_{i=1}^r l(S_i) + \sum_{i=1}^t l(S'_i) \equiv t,
    \]
    where the equivalences were taken modulo \(k\).
    Now, since no vertex of \(x\) is adjacent to a vertex of~\(B^*\),
    then \(S'_i\) is centred at \(B\setminus B^*\)
    for every \(i\in [t]\),
    and hence 
    \[
        k-1 \leq t\leq |B\setminus B^*| = |B| - |B^*| = 2d + k - 1 - |A| - |B^*| =  k-2,
    \]
    a contradiction.
\end{proof}

In what follows, we propose two conjectures which, if true, would enable a further reduction of the multiplicative constant in the upper bounds in Theorem \ref{th:main}, as discussed below.

\begin{conjecture}\label{conj:divisible}
Let \(k\) be a positive integer.
For every $0_k$-graph $G$ we have $\chi'_k(G) \leq k +o(k)$.
\end{conjecture}

If Conjecture \ref{conj:divisible} holds, we could adapt the proof of Theorem \ref{th:main} as follows. 
We suppose that $k$ is odd, since the case $k$ is even is analogous. 
Let \(H\) be a maximal \(0_k\)-subgraph of \(G\)
and consider the subgraph $G'=(V(G), E(G)\setminus E(H))$. 
This graph does not contain any non-empty $k$-divisible subgraph, 
and hence $G'$ is $(2q(k)-2)$-degenerate (while in the proof of Theorem~\ref{th:main}, \(G'\) is \((2q(k) + k - 3)\)-degenerate). 
Now, we can use Conjecture \ref{conj:divisible} to colour \(H\) using $\chi'_k(H) \leq k +o(k)$ colours,
and Lemma~\ref{lem:deg}
to colour \(G'\) with \(\chi'_k(G')\leq 5k + o(k)\) colours,
from which we would require only $6k+o(k)$ colours to colour \(G\). 
The bound for $k$ even would instead be $8k+o(k)$.

\begin{conjecture}\label{conj:primedivisible}
Let \(k\) be a positive integer. For every $0_{q(k)}$-graph $G$ we have $\chi'_k(G) \leq 3k +o(k)$.
\end{conjecture}

If Conjecture \ref{conj:primedivisible} holds, we could adapt the proof of Theorem \ref{th:main} as follows. 
The overall strategy is similar to discussion above. 
Namely, we would identify a maximal $0_{q(k)}$-graph $H$ to be coloured with $\chi'_k(H)\leq3k + o(k)$ colours (using Conjecture~\ref{conj:primedivisible}), and consider the subgraph $G'=(V(G), E(G)\setminus E(H))$.  
This graph does not contain any non-empty $q(k)$-divisible subgraph, therefore, using Theorem \ref{thm:alonfriedlandkalai} together with the same strategy as above, we would show that $G'$ is $(2q(k)-2)$-degenerate, and by Lemma \ref{lem:deg}, we could colour it using $5k+o(k)$ colours. Therefore, we would require $8k+o(k)$ colours to colour \(G\), independently of the parity of $k$. Clearly for $k$ odd, we could simply repeat the reasoning of our proof and obtain a better bound.

\section*{Acknowledgements}
We sincerely thank the editor and reviewers for their time and constructive feedback, which has helped us improve the manuscript. 

This paper is the merge of two independent projects, one of which (involving GB, MB, GC, AD and FG) was initiated as a ``Groupe de recherche JGA'' in Dijon 2024.

This research has been partially supported by
  Coordenação de Aperfeiçoamento  de Pessoal de Nível Superior - Brasil -- CAPES -- Finance Code 001. 
  F.~Botler is supported by CNPq (304315/2022-2), CAPES (88881.973147/2024-01), and FAPESP (2024/14906-6 and  2023/03167-5).  
L.~Colucci is supported by FAPESP (2020/08252-2).  
P. M. V. Neto is supported by FAPESP (2025/13056-1 and 2025/24104-7)
FAPESP is the S\~ao Paulo Research Foundation.  CNPq is the National Council for Scientific and Technological Development of Brazil. M. Bonamy is supported by the ANR project ENEDISC (ANR-24-CE48-7768-01). G. Carenini is supported by the CB European PhD Studentship funded by Trinity College Cambridge. A.~Dumas is supported by the National Research Agency (ANR), France project P-GASE (ANR-21-CE48-0001). F.~Ghasemi is supported by ANR project DIFFERENCE (https://anr.fr/Projet-ANR-20-CE48-0002) and by IUT Sénart-Fontainebleau.

\bibliographystyle{siamplain}
\bibliography{ref}

\end{document}